\documentclass[reqno]{amsart}
\usepackage{amsmath}
\usepackage{amssymb, amsthm, amsfonts, tikz-cd, mathrsfs,mathtools,stmaryrd,enumitem, algorithmic,float,comment,tikz,hyperref}
\usepackage[toc]{appendix}
\usetikzlibrary{backgrounds}
\usetikzlibrary{decorations.pathreplacing}
\usepackage{fullpage}
\usepackage{xcolor}

\usepackage{tikz}
\usetikzlibrary{calc,decorations.pathreplacing}

\hypersetup{
  colorlinks   = true, 
  urlcolor     = blue, 
  linkcolor    = blue, 
  citecolor   = red 
}

\usepackage{comment}

\usepackage[capitalize,nameinlink]{cleveref}
\crefname{theorem}{Theorem}{Theorems}
\crefname{lemma}{Lemma}{Lemmas}
\crefname{claim}{Claim}{Claims}
\crefname{prop}{Proposition}{Propositions}
\crefname{figure}{Figure}{Figures}

\crefname{enumi}{}{}
\Crefname{enumi}{}{}
\crefformat{enumi}{(#2#1#3)}
\Crefformat{enumi}{(#2#1#3)}

\newtheorem{theorem}{Theorem}

\newtheorem{lemma}[theorem]{Lemma}
\newtheorem{claim}[theorem]{Claim}
\newtheorem{corollary}[theorem]{Corollary}

\newtheorem{conj}[theorem]{Conjecture}

\newtheorem{prop}[theorem]{Proposition}

\newtheorem*{claim*}{Claim}

\theoremstyle{remark}
\newtheorem*{remark*}{Remark}
\newtheorem{remark}[theorem]{Remark}

\theoremstyle{definition}
\newtheorem{definition}[theorem]{Definition}

\numberwithin{theorem}{section}

\renewcommand{\phi}{\varphi}

\newcommand{\E}{\mathbb{E}}

\def\1{\mathbbm{1}}

\renewcommand{\le}{\leqslant}
\renewcommand{\ge}{\geqslant}
\renewcommand{\P}{\mathbb{P}}

\newcommand{\R}{\mathbb R}

\newcommand{\lpr}[1]{\left(#1\right)}

\newcommand{\ceil}[1]{\lceil#1\rceil}
\newcommand{\floor}[1]{\lfloor#1\rfloor}
\newcommand{\ang}[1]{\langle#1\rangle}

\newcommand{\bin}{\{0, 1\}}
\newcommand{\supp}{\textup{supp}}

\newcommand{\F}{\mathbb{F}}
\newcommand{\Aff}{\textup{Aff}}

\newcommand{\PG}{\textup{PG}}
\newcommand{\Gr}{\textup{Gr}}
\newcommand{\rank}{\textup{rank}}

\usepackage[
backend=biber,
style=numeric,
maxnames=99,
giveninits=true
]{biblatex}
\renewbibmacro{in:}{%
  \ifentrytype{article}{}{\printtext{\bibstring{in}\intitlepunct}}}

\title
{Affine Subspace Statistics in the Hypercube} 

\author{Zixuan Xu}
\thanks{Massachusetts Institute of Technology. Email: \texttt{zixuanxu@mit.edu}}

\begin{document}

\begin{abstract}
    We study the intersection statistics of affine subspaces in the hypercube $\F_2^n$, motivated by recent work of Alon, Axenovich, and Goldwasser on the intersection statistics of axis-aligned subcubes of an $n$-dimensional cube. Let $d\ge 1$ and $0\le s\le 2^d$ be nonnegative integers. For a subset $A\subseteq \F_2^n$ where $n\ge d$, define $\lambda^*(n,d,s,A)$ to be the fraction of affine $d$-flats in $\F_2^n$ that intersect $A$ at exactly $s$ points. Let $\lambda^*(n,d,s) = \max_{A\subseteq \F_2^n}\lambda^*(n,d,s,A)$ and let $\lambda^*(d,s) = \lim_{n\to \infty}\lambda^*(n,d,s)$. We show that when $s = j\cdot 2^k$ with $j$ odd and $k\ge 1$, we have $\lambda^*(d,s)\to 1-\Theta(2^{-k})$ as $d\to \infty$. This implies that $\lambda^*(d,s)$ is controlled up to constant factors by the $2$-adic valuation of $s$ when $s$ is even. When $s$ is odd, we show that $\lambda^*(d,s)\le \frac{1}{2}$ in contrast to the behavior of axis-aligned subcube statistics. We also present several upper and lower bounds for certain specific values of $s$. 
\end{abstract}

\maketitle
\vspace{-1.2em}

\section{Introduction}

 Recently, work of Alon, Axenovich, and Goldwasser~\cite{alon2024hypercubestatistics} initiated the systematic study of the hypercube statistics problem defined as follows. For a subset $A\subseteq \bin^n$ and nonnegative integers $d\le n$ and $s\le 2^d$, define $\lambda(n,d,s,A)$ as the fraction of $d$-subcubes $Q_d\subseteq Q_n$ such that $|Q_d\cap A| = s$. Then let $\lambda(n,d,s):= \max_{A\subseteq Q_n}\lambda(n,d,s,A)$ and let $\lambda(d,s): =\lim_{n\to \infty} \lambda(n,d,s)$. Note that $\lambda(n,d,s)$ is monotone non-increasing in $n$ by a simple averaging argument, so the limit $\lambda(d,s)$ is well-defined. Except for very specific values of $s$ and a few pairs of small $s$ and $d$, the exact values for $\lambda(d,s)$ are poorly understood.

In this paper, we study the analogous intersection statistics problem for affine subspaces rather than axis-aligned subcubes. From this, we can also obtain results for the hypercube statistics problem. For nonnegative integers $n\ge d\ge 1$ and $s\le 2^d$, define for a subset $A\subseteq \F_2^n$ the affine subspace statistics $\lambda^*(n,d,s,A)$ as the fraction of affine $d$-subspaces $Q\subseteq \F_2^n$ (i.e. $Q = x_0 + U$ where $x_0\in \F_2^n$ and $U\subseteq \F_2^n$ is a $d$-dimensional linear subspace) such that $|Q\cap A| = s$. For convenience, we refer to an affine $d$-subspace as a $d$-flat from now on. Similarly we can define $\lambda^*(n,d,s):= \max_{A\subseteq \F_2^n}\lambda^*(n,d,s,A)$ and $\lambda^*(d,s):= \lim_{n\to \infty}\lambda^*(n,d,s)$. Note that the limit is well-defined because $\lambda^*(n,d,s)$ is also monotone non-increasing. Furthermore, observe that we have
\[\lambda^*(d,s)\le \lambda(d,s).\]
Indeed, taking a construction in the $d$-flat setting and picking the best coordinates would give a construction for $d$-cubes that performs as well as for $d$-flats. The flat setting removes the dependence on the coordinate system which allows the study of the intersection statistics problem under full affine symmetry. Thus, the flat setting provides a basis-invariant version that identifies the obstructions which any improvement in the cube setting must overcome by exploiting the coordinate-dependent structure.

Our main result is to show that for even $s$, we have
\[\lambda^*(d,s) = 1- \Theta(2^{-\nu_2(s)}),\]
where $\nu_2(s)$ denotes the exponent of the largest power of $2$ divisible by $s$ (i.e. the $2$-adic valuation of $s$). More specifically, we show the following theorem.

\begin{theorem}\label{thm:s-even}
 Let $d\ge 1$ and $1<s<2^d$. Suppose $s = j\cdot 2^k$ where $j$ is odd and $k\ge 0$, then we have
\[\lambda^*(d,s)\le  1-\frac{2}{3}(1-2^{-(d-k)})\cdot 2^{-k} + O(2^{-2k})+o_d(1).\]
In particular, for fixed $k$ and $d\to \infty$, we have $\lambda^*(d,s) \to 1-\frac{2}{3}\cdot 2^{-k} + O(2^{-2k})$.
\end{theorem}

On the other hand, a lower bound for $\lambda^*(d,s)$ can be obtained by the same linear algebraic construction presented in Alon--Axenovich--Goldwasser~\cite{alon2024hypercubestatistics}.

\begin{prop}\label{prop:flat-lower-bound}
    Let $d\ge 2$ and $s\le 2^d$ be nonnegative integers. For $s = 2^k\cdot j$ with $j$ odd and $k\ge 1$, we have
    \[\lambda^*(d,s)\ge 1-2^{-k}.\]
    For $s$ odd, we have $\lambda^*(d,s)\ge (1+o_d(1))\cdot 0.2887$.
\end{prop}
For completeness, we include a proof of \cref{prop:flat-lower-bound} in \cref{sec:prelim} which follows similarly to the proof of Theorem 2 in \cite{alon2024hypercubestatistics}. 

Combining \cref{thm:s-even} and \cref{prop:flat-lower-bound}, for $s = j\cdot 2^k$ with $j$ odd and $k\ge 1$, we have
\[\lambda^*(d,s) = 1-\Theta(2^{-k}),\]
which determines $1-\lambda^*(d,s)$ for all even $s$ up to an absolute constant. In particular, this shows that $\lambda^*(d,s)$ is controlled by the $2$-adic valuation of $s$ rather than the magnitude of $s$. This behavior is in sharp contrast to the current known upper bounds for the cube statistics from \cite{alon2024hypercubestatistics} which depend directly on the magnitude of $s$.

Now we present our result for odd $s$.

\begin{theorem}\label{thm:s-odd}
     Let $d\ge 1$ and $1<s<2^d$. Suppose $s$ is odd, then 
     \[\lambda^*(d,s)\le \frac{1}{2}.\]
\end{theorem}
 We note that \cref{thm:s-odd} does not hold for $\lambda(d,s)$ as there are specific values for $d$ and odd $s$ where $\lambda(d,s) > 1/2$ (see \cref{sec:concluding-remarks}). In fact, to prove \cref{thm:s-odd}, we will prove a strictly stronger statement that for any subset $A\subseteq \F_2^n$, the fraction of affine $d$-flats that intersect $A$ at an odd number of points is at most $\frac{1}{2} + o(1)$ as $n\to\infty$. This is not true in the axis-aligned cube setting, as one can construct a set that intersects every $d$-cube in an odd number of points (see \cref{rmk:cube-odd}).

Furthermore, we also determine the exact value of $\lambda^*(d,2^{d-1})$.
\begin{theorem}\label{thm:s-2d-1}
    Let $d > 1$ be an integer. We have
    \[\lambda^*(d,2^{d-1}) = 1-2^{-d}.\]
\end{theorem}
Note that in contrast we have $\lambda(d,2^{d-1}) = 1$, which is achieved by taking the set of points with Hamming weight of the same parity.

Finally, we observe that $\lambda^*(d,s)$ cannot change too drastically between neighboring values. Namely, we show that $\lambda^*(d,s)$ and $\lambda^*(d,s+1)$ differ by a multiplicative factor of at most $e$.
\begin{theorem}\label{thm:one-point}
    For $d>1$ and $2\le s \le 2^{d}$, we have
    \[\frac{\lambda^*(d,s)}{\lambda^*(d,s-1)}\le \lpr{\frac{s}{s-1}}^{s-1}\le e,\]
    and for $0\le s \le 2^d-2$, we have
    \[\frac{\lambda^*(d,s)}{\lambda^*(d,s+1)}\le \lpr{\frac{2^d-s}{2^d-s-1}}^{2^d-s-1}\le e.\]
\end{theorem}

As a direct corollary, we also obtain better lower bounds for certain odd values of $s$.
\begin{corollary}\label{cor:one-point-lb}
    For $d>1$ and $1<s<2^d$ where $s = j\cdot 2^k$ with $j$ odd and $k\ge 1$, we have
    \[\lambda^*(d,s + 1)\ge \frac{1}{e}\cdot (1-2^{-k})\quad \text{and}\quad \lambda^*(d,s - 1)\ge \frac{1}{e}\cdot (1-2^{-k}).\]
\end{corollary}
In particular, for $k\ge 3$ and $d$ large, we have $(1-2^{-k})/e > 0.2887$, improving on \cref{prop:flat-lower-bound} for specific values of $s$. 

\subsection{Comparison to $\lambda(d,s)$}

In this subsection, we compare our results for $\lambda^*(d,s)$ with current known results for $\lambda(d,s)$. We define $c_d$ as the probability that a $d\times d$ matrix over $\F_2$ where each column is an independent random nonzero vector in $\F_2^d$ is nonsingular. Note that we have
\[c_d = \prod_{i = 1}^{d-1}\lpr{1-\frac{2^i-1}{2^d-1}},\]
and $c_d\to 0.2887$ as $d\to \infty$. We also define $c(d,k)$ to be the probability that a $(d-k)\times d$ matrix where each column is an independent random vector in $\F_2^{d-k}$ having rank $d-k$. Note that we have
\[c(d,k)=\prod_{i = 0}^{d-k-1}\lpr{1-\frac{2^i}{2^d}}\]
and for fixed $k$, we have $c(d,k)\to 1 - O(2^{-k})$ as $d\to \infty$.

Now we can recall the current known bounds for $\lambda(d,s)$ in the following theorem.

\begin{theorem}[\cite{alon2024hypercubestatistics}]\label{thm:AAG}
Let $d\ge 2$ and $s\le 2^d$ be nonnegative integers. 
\begin{enumerate}
    \item $\lambda(d,s) = 1$ if and only if $s\in \{0,2^{d-1}, 2^d\}$. \label{item:s=1}
    \item For $1 < s < 2^{d-1}$, we have
    \[\lambda(d,s)\le \lambda(d+2,d,s)\le \lpr{1-\frac{1}{4s-1}}\lpr{1+\frac{1}{d+1}}.\]
    \item For $s = 1$, we have $\lambda(d,1)\ge (1-2^{-d})^{2^d-1}.$ In particular, $
    \lambda(d,1)\to (1+o(1))\frac{1}{e}$ as $d\to \infty$.
    \item For any $1\le s\le 2^d$, we have 
    \[\lambda(d,s)\ge c_d.\]
    \item For $s = 2^k\cdot j$, we have 
    \[\lambda(d,s)\ge c(d,k).\]
\end{enumerate}
\end{theorem}
We remark that \cref{item:s=1} in \cref{thm:AAG} was first proved by Goldwasser and Hansen in \cite{GoldwasserHansen2024Inducibility}. On the lower bound side, we remark that $c_d > c(d,0)$ as $c_d$ is defined to be the nonsingular probability of a $d\times d$ matrix with \emph{nonzero} independent columns. On the other hand, the bounds in \cref{thm:AAG} imply that if $s$ is divisible by a power of $2$ that is $\Theta(s)$, then $\lambda(d,s) = 1- \Theta(1/s)$. Most recently, Bodn{\'a}r and Pikhurko determined in \cite{bodnar_pikhurko_2025_hypercube} three exact values for $\lambda(d,s)$ where $\lambda(d,s)\ne 1$ using the Flag algebra method. Specifically, they proved that $\lambda(3,2) = 8/9$, $\lambda(4,2) = 264/343$, and $\lambda(4,4) = 26/27$. However, obtaining tight bounds for essentially all other values of $d$ and $s$ remains open. 

\subsection{Notations.} For $x,y\in \F_2^n$ where $x = (x_1,\dots, x_n)$ and $y = (y_1,\dots, y_n)$, we define the bilinear form $\ang{x,y} := \sum_{i = 1}^n x_iy_i\pmod{2}$. For a function $f$, we use $f\not\equiv 0$ to denote that $f$ is not identically zero.

\subsection{Paper organization} In \cref{sec:prelim}, we recall definitions and theorems that will be later used in our proofs and include the proof of \cref{prop:flat-lower-bound} for completeness. In \cref{sec:s-even}, we present the proofs of \cref{thm:s-even} and \cref{thm:s-2d-1}. In \cref{sec:s-odd}, we give the proofs of \cref{thm:s-odd} and \cref{thm:one-point}. Finally, in \cref{sec:concluding-remarks}, we include some concluding remarks and open problems.

\medskip\noindent\textbf{Acknowledgments.} The author would like to thank Lisa Sauermann for bringing the work of Alon--Axenovich--Goldwasser~\cite{alon2024hypercubestatistics} to her attention, providing helpful discussions and carefully reading an earlier draft of this note. The author would like to thank Ting-Wei Chao and Dmitrii Zakharov for various useful suggestions and inspiring discussions.

\section{Preliminaries}\label{sec:prelim}

Let $\Gr(n,d)$ denote the set of $d$-dimensional linear subspaces in $\F_2^n$ and let $\Aff(n,d)$ denote the set of $d$-flats in $\F_2^n$. Note that every $Q\in \Aff(n,d)$ can be written as $Q = x_0+U$ for some $x_0\in \F_2^n$ and $U\in \Gr(n,d)$. For a linear subspace $U\subseteq \F_2^n$, we use $U^\perp:=\{x\in \F_2^n\mid \langle x,u\rangle=0\text{ for all }u\in U\}$ to denote its orthogonal complement. Recall the $\F_2$-analogue of $\binom{n}{d}$ is defined as 
\[\binom{n}{d}_2:= \prod_{i=0}^{d-1} \frac{2^{\,n-i}-1}{2^{\,d-i}-1},\]
which is the number of $d$-dimensional linear subspaces in $\F_2^n$. Namely, we have $|\Gr(n,d)| = \binom{n}{d}_2$. Furthermore, we have
\[|\Aff(n,d)| = 2^{n-d}\binom{n}{d}_2.\]
Note that for a subset $A\subseteq \F_2^n$, we can interpret $\lambda^*(n,d,s,A)$ as the probability of a uniformly random $d$-flat $Q\sim \Aff(n,d)$ intersecting $A$ at exactly $s$ points.

Recall that in the introduction, we defined the quantity $c(d,k)$. This is the probability of a random $(d-k)\times d$ matrix with columns being independent random vectors in $\F_2^{d-k}$ having rank $d-k$. Specifically, we have
\[c(d,k)=\prod_{i = 0}^{d-k-1}\lpr{1-\frac{2^i}{2^d}} = \prod_{i = 0}^{d-k-1}(1-2^{i-d}).\]
For our setting, it is more natural to consider the quantity $c_n(d,k)$ defined as follows. Let $B:\F_2^n\to \F_2^{d-k}$ be a surjective linear map. Define
\[c_n(d,k) = \P_{U\in \Gr(n,d)}[\rank(B|_U) = d-k].\]
Then we have 
\[c_n(d,k) = 2^{(d-k)(n-d)}\cdot \frac{\binom{n-d+k}{k}_2}{\binom{n}{d}_2}.\]
Indeed, for $U\in \Gr(n,d)$ to satisfy $\rank(B|_U) = d-k$, we have $\dim (U\cap \ker B) = k$. So we can first choose the $k$-dimensional subspace $W = U\cap \ker B$ from $\ker B$ in $\binom{n-d+k}{k}_2$ many ways. Then for each fixed $W$, the possible $U\in \Gr(n,d)$ with $U\cap \ker B = W$ are the lifts of $\F_2^n/\ker B$ to a $d$-dimensional subspace containing $W$. There are $2^{(d-k)(n-d)}$ such lifts because they are parametrized by the linear maps $\F_2^n/\ker B\to \ker B/W$ where $\dim \F_2^n/\ker B = d-k$ and $\dim \ker B/W = n-d$.
Expanding the $2$-binomial coefficients, we have
\begin{align*}
    \lim_{n\to\infty}c_n(d,k) &= \lim_{n\to\infty} \lpr{\prod_{i = 0}^{d-k-1}(1-2^{i-n})}\lpr{\prod_{i = 0}^{k-1}\frac{1-2^{i-(n-d+k)}}{1-2^{i-n}}}\lpr{\prod_{i = 0}^{d-k-1}(1-2^{i-d})}\\
    &= \lpr{\prod_{i = 0}^{d-k-1}\lim_{n\to\infty}(1-2^{i-n})}\lpr{\prod_{i = 0}^{k-1}\lim_{n\to\infty}\frac{1-2^{i-(n-d+k)}}{1-2^{i-n}}}\prod_{i = 0}^{d-k-1}(1-2^{i-d})\\
    &= \prod_{i = 0}^{d-k-1}(1-2^{i-d}) = c(d,k)
\end{align*}
since the factors in the first two finite products both tend to $1$ as $n\to\infty$.

We will also use the discrete Fourier transform over $\F_2^n$. For a function $f:\F_2^n\to \R$, it is more convenient for us to use the \emph{unnormalized} Fourier transform $\hat{f}: \F_2^n\to \R$ defined as
\[\hat{f}(\xi) = \sum_{x\in \F_2^n}f(x)(-1)^{\ang{\xi, x}}.\]
Thus we have the inversion formula
\[f(x) = 2^{-n}\sum_{\xi\in \F_2^n}\hat{f}(\xi)(-1)^{\ang{\xi,x}}.\]
Recall Parseval's identity as 
\[\sum_{\xi\in \F_2^n}\hat{f}(\xi)^2 = 2^n\sum_{x\in \F_2^n}f(x)^2.\]
For functions $f,g:\F_2^n\to \R$, the convolution of $f$ and $g$ is defined as 
\[(f*g)(x) := \sum_{y\in \F_2^n}f(y)g(x-y),\]
and recall that we have $\widehat{f*g}(\xi) = \hat{f}(\xi)\hat{g}(\xi)$.

We will also use some standard facts about the projective space over $\F_2$. Let $\PG(n,2)$ denote the $n$-dimensional projective space over $\F_2$, consisting of the nonzero vectors in $\F_2^{n+1}$. Namely, we identify $\PG(n,2) \cong \F_2^{n+1}\setminus \{0\}$. Moreover, note that a projective $k$-dimensional subspace in $\PG(n,2)$ corresponds to a $(k+1)$-dimensional subspace in $\F_2^{n+1}$. We recall the definition of a blocking set.

\begin{definition}[Blocking set]\label{def:blocking-set}
A set $B\subseteq \PG(n,2)$ is called a blocking set with respect to $k$-subspaces if $B\cap U\ne \varnothing$ for all $k$-dimensional projective subspaces $U\subseteq \PG(n,2)$.
\end{definition}

We will use the Bose--Burton theorem that gives a lower bound for the size of a blocking set in $\PG(n,2)$.

\begin{theorem}[Bose--Burton~\cite{BoseBurton1966}]\label{thm:bose-burton}
Let $B\subseteq \PG(n,2)$ be a blocking set with respect to $k$-subspaces. Then 
\[|B|\ge |\PG(n-k,2)| = 2^{n-k+1}-1,\]
where equality holds when $B$ is an $(n-k)$-dimensional projective subspace.
\end{theorem}

\subsection{Lower bound construction.} For completeness, we include a proof of \cref{prop:flat-lower-bound} which is similar to the proof of Theorem 2 in \cite{alon2024hypercubestatistics}.

\begin{proof}[Proof of \cref{prop:flat-lower-bound}]
    Let $s = j\cdot 2^k$ for $j$ odd and $k\ge 0$. Let $B:\F_2^n\to \F_2^{d-k}$ be a surjective linear map and let $S\subseteq \F_2^{d-k}$ be a subset of size $|S| = j$. Then consider the subset
    \[A := B^{-1}(S) = \{x\in \F_2^{n}\mid B(x)\in S\}\subseteq \F_2^n.\]
    We show that $\lambda^*(n,d,s,A) \ge c_n(d,k)$. Let $Q = x_0+U$ be a $d$-flat where $x_0\in \F_2^n$ and $U\in \Gr(n,d)$. Then note that 
    \[|Q\cap A| = 2^{d-\rank(B|_U)}\cdot |S\cap B(Q)|,\]
    where $B(Q) = \{B(x) \mid x\in Q\}$. It is clear that $|Q\cap A| = s$ if $\rank(B|_U) = d-k$. (In fact, if $k > 0$, we have $|Q\cap A| = s$ if and only if $\rank(B|_U) = d-k$.) Thus by definition, we have 
    \[\lambda^*(n,d,s,A) = \P[|Q\cap A| = s]\ge \P[\rank(B|_U) = d-k] = c_n(d,k).\]
    Thus, we can conclude that $\lambda^*(d,s) \ge \lim_{n\to \infty} \lambda^*(n,d,s,A) = \lim_{n\to \infty}c_n(d,k) = c(d,k)$.
\end{proof}

\begin{remark}
    We remark that for the axis-aligned cube statistics, one can choose $B$ to be a linear map defined by a $(d-k)\times n$ matrix of rank $d-k$ with all columns being nonzero vectors in $\F_2^{d-k}$. Then the lower bound can be improved to the probability of a $(d-k)\times d$ matrix with each column being an independent random nonzero vector in $\F_2^{d-k}$ having rank $d-k$ as remarked in \cite{alon2024hypercubestatistics}. In the affine flat setting, a $(d- k)\times d$ matrix arises from restricting a surjective linear map $B$ to a random $d$-subspace $U$. Since the basis vectors of $U$ may lie in $\ker B$, the induced columns are not forced to be nonzero, so one cannot get such an improvement.
\end{remark}

\section{$s$ even}\label{sec:s-even}

In this section, we present the proofs of \cref{thm:s-2d-1} and \cref{thm:s-even}, giving upper bounds for $\lambda^*(d,s)$ when $s$ is even. Let us begin with the simpler special case where $s = 2^{d-1}$.

\subsection{The case $s = 2^{d-1}$}

From now on, fix $d> 1$ and $n\ge d$. In the case where $s = 2^{d-1}$, we observe a clear distinction between the cube statistics and flat statistics. Recall that when $s = 2^{d-1}$, in the cube statistics setting, we have $\lambda(n,d,s) = 1$ which is achieved by taking the set $A\subseteq \F_2^n$ containing all the points in $\F_2^n$ with an even number of $1$'s or all the points with an odd number of $1$'s. Equivalently, $A$ is the hyperplane defined by the linear equation $\ang{x, \vec{1}} = 0$ or $\ang{x, \vec{1}} = 1$. Additionally, it was remarked in \cite{alon2024hypercubestatistics} that one can get more constructions of $A\subseteq \F_2^n$ by choosing an $(n-d+t)$-subcube $Q$ for $1\le t\le d-1$ and replacing the set $A\cap Q$ with its complement.

First we compute the flat statistics for the set $A$ of the points on the hyperplane $\ang{x, \vec{1}} = 0$. For convenience, let $L(x) = \ang{x, \vec{1}}$. For a $d$-flat $Q = x_0+U\subseteq \F_2^n$ where $U\in \Gr(n,d)$, note that if $L|_U \not\equiv 0$, then $|Q\cap A| = 2^{d-1}$; if $L_U\equiv 0$, then $|Q\cap A|\in\{0,2^d\}$. Thus we have that 
\[\lambda^*(n,d,s,A) = 1- \P_{U\sim \Gr(n,d)}[U\subseteq \ker L] = 1- \frac{\binom{n-1}{d}_2}{\binom{n}{d}_2} = 1- \prod_{i = 0}^{d-1}\frac{2^{n-1-i}-1}{2^{n-i}-1},\]
where the probability $\P_{U\sim \Gr(n,d)}[U\subseteq \ker L]$ is over a uniformly random element $U$ in $\Gr(n,d)$. In particular, we have
\[\lambda^*(d,s) \ge \lim_{n\to \infty}\lpr{1-\prod_{i = 0}^{d-1}\frac{2^{n-1-i}-1}{2^{n-i}-1}}= 1-2^{-d} < 1.\]

We show that this is optimal in the following lemma using a second moment argument.

\begin{lemma}
    For $d> 1$ and $n \ge d$, for any $A\subseteq \F_2^n$ we have
    \[\lambda^*(n,d,2^{d-1},A)\le 1- \frac{\binom{n-1}{d}_2}{\binom{n}{d}_2}.\]
    Consequently,
    \[\lambda^*(d,2^{d-1}) = 1 - 2^{-d}.\]
\end{lemma}

\begin{proof}
First note that to sample a uniformly random $d$-flat $Q$ from $\Aff(n,d)$, one can first sample a uniformly random linear $d$-subspace $U\in \Gr(n,d)$ and then sample a uniformly random $x_0\in \F_2^n$. Then taking $Q = x_0+U$ gives a uniformly random $d$-flat. In this proof, we will consider a uniformly random $Q\in \Aff(n,d)$ sampled in this way.

Let $f=1_A$ and consider $h:=1-2f$. Note that $h:\F_2^n\to \{\pm 1\}$ where $h(x) = 1$ if $x\not\in A$ and $h(x) = -1$ if $x\in A$. For a linear $d$-subspace $U\in \Gr(n,d)$ define
\[H_U(x):=\sum_{u\in U} h(x+u)=(h*1_U)(x),\]
and notice that for a $d$-flat $Q=x_0+U$ with $x_0\in \F_2^n$ we have
\[\sum_{y\in Q} h(y)=H_U(x_0)=2^d-2|A\cap Q|.\]
In particular, $|A\cap Q|=2^{d-1}$ if and only if $H_U(x_0)=0$. 

Now we fix $U\in \Gr(n,d)$ and let $x\in \F_2^n$ be a uniformly random element. Since $|H_U(x)|\le 2^d$ for all $x\in \F_2^n$, we have $H_U(x)^2\le 2^{2d}1_{H_U(x)\ne 0}$ and thus 
\[\P[H_U(x)=0]\le 1-\frac{\E[H_U(x)^2]}{2^{2d}}.\]
Note that for $U\in \Gr(n,d)$, we have $\widehat{1_U}(\xi) = |U|\cdot 1_{\xi\in U^{\perp}}$, so we have 
\[\widehat{H_U}(\xi) = \widehat{(h * 1_U)}(\xi) =\widehat{h}(\xi)\cdot\widehat{1_U}(\xi)
=\widehat{h}(\xi)\cdot |U|\cdot 1_{\{\xi\in U^\perp\}}
=2^d \cdot \widehat{h}(\xi)\cdot 1_{\{\xi\in U^\perp\}}.\]
Applying Parseval's identity gives
\[\E[H_U(x)^2]=2^{-n}\sum_x H_U(x)^2 =2^{-2n}\sum_{\xi\in \F_2^n}\widehat{H_U}(\xi)^2 =2^{2d-2n}\cdot \sum_{\xi\in U^\perp}\widehat{h}(\xi)^2.\]
Thus we have
\[ \P[H_U(x)=0] \le 1-\frac{2^{2d-2n}}{2^{2d}}\sum_{\xi\in U^\perp}\widehat{h}(\xi)^2 =1-2^{-2n}\sum_{\xi\in U^\perp}\widehat{h}(\xi)^2. \]

Recall that we can sample a uniformly random $Q\in \Aff(n,d)$ by first choosing $U\in \Gr(n,d)$ uniformly at random, and then choosing $x$ uniformly from $\F_2^n$ and setting $Q=x+U$. So we have
\begin{align*}
    \lambda^*(n,d,2^{d-1},A) &= \P_{Q\sim \Aff(n,d)}[|A\cap Q| = 2^{d-1}] = \E_U[\P_x[|A\cap (x+U)|= 2^{d-1}]]\\
    & =\E_U[\P_x[H_U(x)=0]] \le 1-2^{-2n}\sum_{\xi\in \F_2^n}\widehat{h}(\xi)^2\cdot \P_U[\xi\in U^\perp],
\end{align*}
where the last inequality comes from exchanging the order of expectation and summation.

For a uniformly random $U\in \Gr(n,d)$, if $\xi\ne 0$, we have 
\[\P_U[\xi\in U^\perp] = \frac{\binom{n-1}{d}_2}{\binom{n}{d}_2};\]
and if $\xi = 0$, we have $\P_U[0\in U^{\perp}] = 1$. Therefore, we have $\P_U[\xi\in U^\perp] \ge \binom{n-1}{d}_2/\binom{n}{d}_2$.

Furthermore, by Parseval's identity on $h$, we have
\[\sum_{\xi\in \F_2^n}\hat{h}(\xi)^2 = 2^n\sum_{x\in \F_2^n}h(x)^2 = 2^n\cdot 2^n = 2^{2n}.\]
Thus, we can conclude that 
\[\lambda^*(n,d,2^{d-1},A) \le 1- 2^{-2n} \cdot \frac{\binom{n-1}{d}_2}{\binom{n}{d}_2}\cdot  \sum_{\xi\in \F_2^n}\hat{h}(\xi)^2 = 1- 2^{-2n} \cdot 2^{2n}\cdot \frac{\binom{n-1}{d}_2}{\binom{n}{d}_2} = 1- \frac{\binom{n-1}{d}_2}{\binom{n}{d}_2}.\]
Consequently, we have
\[\lambda^*(d,2^{d-1}) \le \lim_{n\to \infty}\lpr{1- \frac{\binom{n-1}{d}_2}{\binom{n}{d}_2}} = 1-2^{-d}.\]
Combined with the lower bound construction presented previously, we can conclude that $\lambda^*(d,2^{d-1}) = 1-2^{-d}$.
\end{proof}

\begin{remark}
    In fact, the proof of \cref{thm:s-2d-1} shows that for any subset $A\subseteq \F_2^n$ with size $|A|\ne 2^{n-1}$ we have $\lambda^*(n,d,2^{d-1}, A)<1 - \binom{n-1}{d}_2/\binom{n}{d}_2$. Indeed, since $\hat{h}(0) = 2^n-2|A|$, we have $\hat{h}(0)^2 > 0$ if $|A|\ne 2^{n-1}$.
\end{remark}

\subsection{The case $s = j\cdot 2^k$}

In this subsection, we present the proof of \cref{thm:s-even}. We first prove a weaker upper bound for $\lambda^*(d,s)$ for all $s = j\cdot 2^k$ with $j$ odd and $k\ge 0$ stated in the following proposition. In particular, the following proposition already gives the correct asymptotics for $1-\lambda^*(d,s)$ for even $s$.

\begin{prop}\label{prop:s-even}
     Let $d\ge 1$ and $1<s<2^d$ and $n\ge d+1$. Suppose $s = j\cdot 2^k$ where $j$ is odd and $k\ge 0$, then for any $A\subseteq \F_2^n$ we have 
    \[\lambda^*(n,d,s,A)\le 1 - \frac{2^{d-k}-1}{2^{d+1}-1}.\]
\end{prop}

Note that \cref{prop:s-even} already gives a slightly weaker upper bound on $\lambda^*(d,s)$ since we have
\[\lambda^*(d,s)  = \lim_{n\to \infty} \max_{A\subseteq \F_2^n}\lambda^*(n,d,s,A) \le 1 - \frac{2^{d-k}-1}{2^{d+1}-1} = 1-2^{-(k+1)} + O(2^{-(d-k)}).\] 
In particular, this already determines the correct asymptotics for $1-\lambda^*(d,s)$ where $s = j\cdot 2^k$ with $j$ odd and $k\ge 1$. Our proof is an averaging argument over $(d+1)$-flats.

\begin{proof}[Proof of \cref{prop:s-even}]
    Fix $A\subseteq \F_2^n$ and $s = j\cdot 2^k$ where $j$ is odd and $k \ge 0$. For each $(d+1)$-flat $F\in \Aff(n,d+1)$, we define the local statistics
    \[\lambda_F:= \P_{\substack{Q\subset F\\ \dim Q=d}}[|Q\cap A| = s],\]
    where the probability is over a uniformly random $d$-flat $Q$ contained in $F$. Note that by definition, we have
    \[\lambda^*(n,d,s,A) = \E_{F\sim \Aff(n,d+1)}[\lambda_F],\]
    where the expectation is over a uniformly random $(d+1)$-flat $F$. Therefore, it suffices to upper bound $\lambda_F$ for every $F\in \Aff(n,d+1)$. We will show that for every fixed $F\in \Aff(n,d+1)$, we have 
    \[\lambda_F\le 1- \frac{2^{d-k}-1}{2^{d+1}-1}.\]

    Fix any $F\in \Aff(n,d+1)$ and let $S = F\cap A$. From now on, we identify $F$ with $\F_2^{d+1}$. For $\xi\ne 0\in \F_2^{d+1}$, we define the hyperplanes $H_{\xi, 0} = \{x\in \F_2^{d+1} \mid \ang{\xi, x} = 0\}$ and $H_{\xi, 1} = \{x\in \F_2^{d+1} \mid \ang{\xi, x} = 1\}$.
    Then note that $H_{\xi,0}$ and $H_{\xi,1}$ are two disjoint parallel $d$-flats whose union is $\F_2^{d+1}$. Thus we have
    \[|S| = |S\cap H_{\xi,0}| + |S\cap H_{\xi,1}|.\]
    Now we bound $\lambda_F$ by distinguishing two cases depending on whether $|S| = 2s$. 

    \medskip\noindent\emph{Case 1: $|S|\neq 2s$.} In this case, notice that in every possible hyperplane direction $\xi\ne 0 \in \F_2^{d+1}$, at most one of $|S\cap H_{\xi,0}|$ and $|S\cap H_{\xi,1}|$ can be equal to $s$. Therefore, it is clear that $\lambda_F\le 1/2$. Since $1- \frac{2^{d-k}-1}{2^{d+1}-1} \ge 1/2$ for all $d> k\ge 0$, this case is proved.

    \medskip\noindent\emph{Case 2: $|S|= 2s$.} In this case, for every hyperplane direction $\xi\ne 0 \in \F_2^{d+1}$, either both $|S\cap H_{\xi,0}|$ and $|S\cap H_{\xi,1}|$ are equal to $s$, or neither of them are equal to $s$. Therefore it suffices to show a lower bound for the number of $\xi \ne 0$ such that $|S\cap H_{\xi,0}|$ and $|S\cap H_{\xi,1}|$ are both not equal to $s$. For convenience, we denote $B_F = \{\xi\ne 0\in \F_2^{d+1}\mid |S\cap H_{\xi,0}|\ne s\}$ and note that 
    \[\lambda_F = 1 - \frac{|B_F|}{2^{d+1}-1},\]
    where the $2^{d+1}-1$ in the denominator is the number of directions $\xi\ne 0 \in \F_2^{d+1}$ of the hyperplanes. Let $f := 1_S$ be the indicator function of $S$ in $\F_2^{d+1}$.
    \begin{claim}\label{claim:B-supp}
        We have $B_F = \supp(\hat{f})\setminus\{0\}$.
    \end{claim}

    \begin{proof}
        Notice that by definition, for $\xi\ne 0\in \F_2^{d+1}$ we have
        \[\hat{f}(\xi) = \sum_{x\in \F_2^{d+1}}1_S(x)(-1)^{\ang{\xi,x}} = |S\cap H_{\xi,0}| - |S\cap H_{\xi,1}|.\]
        Then indeed $|S\cap H_{\xi,0}| = |S\cap H_{\xi,1}| = s$ if and only if $\hat{f}(\xi) = 0$. Thus we can conclude that $B_F = \supp(\hat{f})\setminus\{0\}$.
    \end{proof}

    Now it suffices to lower bound $\supp(\hat{f})$. The crucial observation is that $\supp(\hat{f})$ is a blocking set with respect to projective $(k+1)$-subspaces in $\PG(d,2)$, as shown by the following claim.

    \begin{claim}\label{claim:divisibility}
        Let $L\subseteq F\cong \F_2^{d+1}$ be a $t$-dimensional linear subspace for some $0\le t\le d$. If $\hat{f}(\xi) = 0$ for all $\xi \in L\setminus\{0\}$, then every coset of $L^\perp$ contains exactly $|S|/2^t$ points of $S$. In particular, $|S|$ is divisible by $2^t$.
    \end{claim}

    \begin{proof}
        Consider the orthogonal complement $L^\perp$ of codimension $t$ and note that $F$ is a disjoint union of $2^t$ cosets of $L^\perp$. For each coset $a + L^\perp$ where $a\in F/L^\perp$, we define 
        \[c(a) = |S\cap (a+L^\perp)|.\]
        Note that we have $|S| = \sum_{a\in F/L^\perp}c(a)$.

        Fix $\xi\in L$, note that the value $(-1)^{\ang{\xi, x}}$ is constant on each coset of $L^\perp$. Indeed, for $x$ and $x'$ lying in the same coset, we have $x-x'\in L^\perp$. So we must have $\ang{\xi,x'} = \ang{\xi,x}+\ang{\xi,x'-x} = \ang{\xi,x}$. Therefore, we can rewrite $\hat{f}(\xi)$ as
        \[\hat{f}(\xi) = \sum_{a\in F/L^{\perp}}\sum_{x\in S\cap (a+L^\perp)}(-1)^{\ang{\xi, x}} = \sum_{a\in F/L^\perp}c(a)(-1)^{\ang{\xi,a}}.\]
        By assumption, for every $\xi\ne 0\in L$, we have
        \[\hat{f}(\xi) =  \sum_{a\in F/L^\perp}c(a)(-1)^{\ang{\xi,a}} = 0.\]
        On the other hand, $\hat{f}(0) =  \sum_{a\in F/L^\perp}c(a) = |S|$. Thus, for any fixed $a_0\in F/L^\perp$ we have
        \begin{align*}
            |S| &= \sum_{\xi\in L}\hat{f}(\xi)(-1)^{\ang{\xi,a_0}} =\sum_{\xi\in L} \sum_{a\in F/L^\perp}c(a)(-1)^{\ang{\xi,a}}(-1)^{\ang{\xi,a_0}} = \sum_{\xi\in L} \sum_{a\in F/L^\perp}c(a)(-1)^{\ang{\xi,a+a_0}}\\
            &= \sum_{a\in F/L^\perp}c(a)\sum_{\xi\in L}(-1)^{\ang{\xi,a+a_0}}.
        \end{align*}
        Note that the inner sum $\sum_{\xi\in L}(-1)^{\ang{\xi,a+a_0}}$ is $2^t$ if $a = a_0$ and $0$ otherwise. Therefore we have
        \[|S| =c(a_0)\cdot 2^t.\]
        Furthermore, the above holds for any arbitrary $a_0\in F/L^\perp$. So, we have $c(a) = |S|/2^t$ for every $a\in F/L^\perp$. On the other hand, since $c(a) = |S\cap (a+L^\perp)|$ is an integer, we have $|S|$ must be divisible by $2^t$.
    \end{proof}

     For any $(k+2)$-dimensional subspace $L$, by \cref{claim:divisibility}, we must have $B_F\cap (L\setminus \{0\})\ne \varnothing$ since $2s$ is not divisible by $2^{k+2}$. In particular, $B_F$ is a blocking set with respect to $(k+1)$-subspaces in $\PG(d,2)$. Thus by \cref{thm:bose-burton}, we have
    \[|B_F|= |\supp(\hat{f})\setminus \{0\}|\ge 2^{d-k}-1.\]
    Thus we can conclude 
    \[\lambda_F\le 1- \frac{2^{d-k}-1}{2^{d+1}-1},\]
    for the case where $|S| = 2s$. This concludes the proof
\end{proof}

Now we are ready to prove \cref{thm:s-even}. The proof proceeds by bootstrapping the upper bound given in \cref{prop:s-even}. More specifically, in the proof of \cref{prop:s-even}, we upper bound $\lambda^*(d,s)$ by upper bounding the local statistics $\lambda_F$ for $(d+1)$-flats $F$. We distinguish two cases based on whether $|A\cap F| = 2s$ and then we simply took the maximum of the two upper bounds instead of computing the average of $\lambda_F$ over all $F\in \Aff(n,d+1)$. The point is that we can upper bound the fraction of $F\in \Aff(n,d+1)$ such that $|A\cap F| = 2s$ using $\lambda^*(d+1,2s)$. Since \cref{prop:s-even} gives an upper bound for all $d'$ and even $s'$, we can iterate this argument and further improve on the constants for $1-\lambda^*(d,s)$.

\begin{proof}[Proof of \cref{thm:s-even}]
    We will show that 
    \begin{equation}
        \lambda^*(d,s)\le \frac{1}{2}\sum_{m=0}^{\infty}\prod_{i=0}^{m-1}\lpr{\frac{1}{2}-\frac{2^{d-k}-1}{2^{d+i+1}-1}}.\label{eq:bootstrap-final}
    \end{equation}
    Note that the above expression evaluates to $1-\frac{2}{3}(1-2^{-(d-k)})\cdot 2^{-k} + O(2^{-2k}) + o_d(1)$. For convenience, let $a_i =\frac{2^{d-k}-1}{2^{d+i+1}-1}$ and note that we have 
    \[a_i =\frac{2^{d-k}-1}{2^{d+i+1}-1} = (1-2^{-(d-k)})2^{-k-i-1} + O(2^{-2k - (d-k) - 2i}).\]
    So we have 
    \[\lambda^*(d,s)\le \frac{1}{2}\sum_{m = 0}^\infty\prod_{i=0}^{m-1}\frac{1}{2}(1-2a_i) = \frac{1}{2}\sum_{m = 0}^\infty 2^{-m}\prod_{i=0}^{m-1}(1-2a_i).\]
    Since for all $m$, we have $\sum_{i=0}^{m-1}a_i\le 2^{-k+1}$, we have 
    \[\prod_{i=0}^{m-1}(1-2a_i) = 1-2\sum_{i = 0}^{m-1}a_i+O(2^{-2k}).\]
    Thus we have
    \[ \lambda^*(d,s)\le  \frac{1}{2}\sum_{m = 0}^\infty 2^{-m}-  \sum_{m = 0}^\infty 2^{-m}\sum_{i = 0}^{m-1}a_i + O(2^{-2k}).\]
    Note that the first sum is equal to $2$ and by changing the order of summation, the second sum evaluates to 
    \[\sum_{i=0}^\infty a_i\sum_{m = i+1}^\infty 2^{-m} = \sum_{i = 0}^\infty a_i\cdot 2^{-i} = (1-2^{-(d-k)})2^{-k-1}\sum_{i = 0}^\infty 4^{-i} = \frac{2}{3}(1-2^{-(d-k)})2^{-k}.\]
    Thus we can conclude that for fixed $k$ we have 
    \[\lambda^*(d,s)\le 1-\frac{2}{3}(1-2^{-(d-k)})\cdot 2^{-k} + O(2^{-2k}) + o_d(1).\]
    as desired.

    Now we show \cref{eq:bootstrap-final}. For each $t\ge 0$, we define $s_t := 2^t\cdot s = j\cdot 2^{k+t}$ and define $\lambda_t:= \lambda^*(d+t, s_t) = \lambda^*(d+t, j\cdot 2^{k+t})$. We will derive a recurrence for $\lambda_t$ by applying the same analysis in the proof of \cref{prop:s-even} to $(d+t+1)$-flats and $s_t$. Fix $t\ge 0$ and $n\ge d+t+1$, and let $A\subseteq \F_2^n$. Define for each $(d+t+1)$-flat $F\in \Aff(n,d+t+1)$ the local statistics
    \[\lambda_F^{(t)} := \P_{\substack{Q\subseteq F\\ \dim Q = d+t}}[|Q\cap A| = s_t],\]
    where the probability is over a uniformly random $(d+t)$-flat contained in $F$. It follows by averaging over all $(d+t+1)$-flats that we have
    \[\lambda^*(n,d+t,s_t, A) = \E[\lambda_F^{(t)}].\]
    Now fix $F\in \Aff(n,d+t+1)$ and let $S: = F\cap A$. We also distinguish two cases depending on whether $|S| = 2s_t$. By the same argument as in proof of \cref{prop:s-even}, we obtain the following: 
    \begin{itemize}
        \item If $|S| \ne 2s_t$, we have $\lambda_F^{(t)}\le 1/2$.
        \item If $|S| = 2s_t$, we have
        \[\lambda_F^{(t)}\le 1 - \frac{2^{(d+t)-(k+t)}-1}{2^{d+t+1}-1} = 1 - \frac{2^{d-k}-1}{2^{d+t+1}-1}.\]
    \end{itemize}
    Thus we have
    \[\lambda^*(n,d+t, s_t,A)\le \frac{1}{2}\P[|F\cap A|\ne 2s_t] + \lpr{1 - \frac{2^{d-k}-1}{2^{d+t+1}-1}}\P[|F\cap A| = 2s_t],\]
    where $F$ is a uniformly random $(d+t+1)$-flat. Note that since $\P[|F\cap A| = 2s_t] = \lambda^*(n,d+t+1,s_{t+1},A)$, we obtain
    \[\lambda^*(n,d+t, s_t,A)\le \frac{1}{2} + \lpr{\frac{1}{2} - \frac{2^{d-k}-1}{2^{d+t+1}-1}}\lambda^*(n,d+t+1,s_{t+1},A).\]
    For convenience, let $c_t := \frac{1}{2} - \frac{2^{d-k}-1}{2^{d+t+1}-1}$. Then taking the maximum over all $A\subseteq \F_2^n$ and taking $n\to\infty$, we get the recurrence
    \[\lambda_t\le \frac{1}{2} + c_t\lambda_{t+1}.\]
    Note that $0\le \lambda_t\le 1$ for all $t\ge 0$ and we have $\lambda_0 = \lambda^*(d,s)$, so it now suffices to show that \cref{eq:bootstrap-final} is the solution to $\lambda_0$.

    For any integer $m\ge 1$, by repeated substitution, we get 
    \[\lambda_0\le \frac{1}{2}\sum_{\ell=0}^{m-1}\prod_{i=0}^{\ell-1}c_i + \left(\prod_{i=0}^{m-1}c_i\right)\lambda_m.\]
    To see this, we induct on $m$. The base case $m=1$ is exactly the recurrence for $\lambda_0$.
    Assuming the inequality holds for $m$, we apply the recurrence for $\lambda_m$ and obtain
    \begin{align*}
        \lambda_0 &\le \frac{1}{2}\sum_{\ell=0}^{m-1}\prod_{i=0}^{\ell-1}c_i + \left(\prod_{i=0}^{m-1}c_i\right)\lambda_m
        \le \frac{1}{2}\sum_{\ell=0}^{m-1}\prod_{i=0}^{\ell-1}c_i + \left(\prod_{i=0}^{m-1}c_i\right)\left(\frac{1}{2}+c_m\lambda_{m+1}\right)\\
        &= \frac{1}{2}\sum_{\ell=0}^{m}\prod_{i=0}^{\ell-1}c_i + \left(\prod_{i=0}^{m}c_i\right)\lambda_{m+1}.
    \end{align*}

    Note that for all $t\ge 0$, we have
    \[0<c_t=\frac{1}{2}-\frac{2^{d-k}-1}{2^{d+t+1}-1}<\frac{1}{2}.\]
    Since $0\le \lambda_m\le 1$, we have
    \[0\le \lpr{\prod_{i=0}^{m-1}c_i}\lambda_m\le \prod_{i=0}^{m-1}c_i\le 2^{-m}\to 0.\]
    Thus, taking $m\to\infty$, we obtain
    \[
    \lambda^*(d,s)=\lambda_0\le \frac{1}{2}\sum_{m=0}^{\infty}\prod_{i=0}^{m-1}c_i,
    \]
    which is exactly \cref{eq:bootstrap-final}.
    
\end{proof}

Together with the lower bound in \cref{prop:flat-lower-bound}, for $s = j\cdot 2^k$ for odd $j$ and $k\ge 1$, we have
\[\lambda^*(d,s) = 1-\Theta(2^{-k}).\]
This shows that in the flat model, the $2$-adic valuation $\nu_2(s)$ controls the statistics $\lambda^*(d,s)$. However, the behavior is likely different for the axis-aligned cube statistics as illustrated in the case where $s$ is odd.

\section{$s$ odd}\label{sec:s-odd}

By \cref{prop:flat-lower-bound} we know that if $s>1$ is odd, then $\lambda^*(d,s)\ge c(d,0)$ where $c_d\to 0.2887$ as $d\to \infty$. On the other hand, \cref{thm:s-even} implies that when $s$ is odd, i.e. $k = 0$, we have $\lambda^*(d,s)\le 1 - \frac{2^d-1}{2^{d+1}-1}$ which tends to $1/2$ as $d\to \infty$. In fact, for $s$ odd, we have an alternative simpler proof of an upper bound of $1/2$. We prove \cref{thm:s-odd} by proving the following stronger statement.

\begin{lemma}\label{lem:s-odd}
    For any $A\subseteq \F_2^n$ and $1\le d < n$, for a uniformly random $d$-flat $F\in\Aff(n,d)$ we have
    \[\P[|A\cap F| \text{ is odd}] \le \frac{1}{2} + \frac{1}{2(2^{n-d+1}-1)}.\]
\end{lemma}

\begin{proof}
    Fix $U\in \Gr(n,d-1)$ and note that there are $2^{n-d+1}$ affine $(d-1)$-flats parallel to $U$ given by $x+U$ where $x\in \F_2^n/U$. Then for convenience, define for each $x\in \F_2^n/U$ the value $p(x)\in \bin$ to be the parity of $|A\cap (x+U)|$, i.e. $|A\cap (x+U)|\pmod{2}$. 

    Notice that for a $d$-flat $F = (x+U)\sqcup (y+U)$ for distinct $x,y\in \F_2^n/U$, we have that $|A\cap F|\equiv p(x)+p(y) \pmod{2}$ and thus $|A\cap F|$ is even if and only if $p(x) = p(y)$. Let $m$ be the number of $x\in \F_2^n/U$ such that $p(x) = 1$, i.e. the number of affine $(d-1)$-flats parallel to $U$ that intersect $A$ in an odd number of points and for convenience let $N = 2^{n-d+1}$. We sample a uniformly random $d$-flat $F$ by first sampling a uniformly random $U\in\Gr(n,d-1)$, then sample distinct $x,y\in \F_2^n/U$, and setting $F = (x+U)\sqcup (y+U)$. Then we have
    \begin{align*}
        \P[|A\cap F|\text{ is even} \mid U] &= \frac{1}{\binom{N}{2}}\cdot \lpr{\binom{m}{2} + \binom{N-m}{2}} = \frac{m(m-1) + (N-m)(N-m-1)}{N(N-1)} \\
        &= \frac{N(N-1)-2m(N-m)}{N(N-1)}
        \ge \frac{1}{2} - \frac{1}{2(N-1)}.
    \end{align*}
    Thus we have 
    \[\P[|A\cap F|\text{ is odd} \mid U]\le \frac{1}{2} + \frac{1}{2(2^{n-d+1}-1)}.\]
    Averaging over all possible $U$, we obtain the desired result.
\end{proof}
\cref{thm:s-odd} follows directly since for odd $s$ and any $A\subseteq \F_2^n$ we have
\[\lambda^*(n,d,s,A) = \P_{F\in \Aff(n,d)}[|A\cap F| = s]\le \P_{F\in \Aff(n,d)}[|A\cap F| \text{ is odd}] \le \frac{1}{2} + \frac{1}{2(2^{n-d+1}-1)}.\]
Taking $n\to \infty$ shows that $\lambda^*(d,s)\le 1/2$.

\begin{remark}\label{rmk:cube-odd}
Note that \cref{lem:s-odd} does not hold in the axis-aligned cube setting. In particular, one can construct $A\subseteq \F_2^n$ such that every $d$-subcube intersects $A$ in an odd number of points. Consider the degree-$d$ symmetric polynomial over $\F_2$ defined by
\[f(x) = \sum_{T\in \binom{[n]}{d}}\prod_{i\in T}x_i.\]
Let $A = \{x\in \F_2^n\mid f(x) = 1\}$. Fix any $d$-subcube $Q\subseteq \F_2^n$ and note that it can be written as 
\[Q = \{a+u_1e_{i_1}+\dots + u_de_{i_d}\mid u_1,\dots, u_d\in \F_2\}\]
with distinct $i_1,\dots, i_d\in [n]$ and $a\in \F_2^n$ satisfying $a_{i_1} = \dots = a_{i_d} = 0$. We claim that in $\F_2$, we have
\[\sum_{u\in \F_2^d}f(a+u_1e_{i_1}+\dots + u_de_{i_d}) = 1.\]
For convenience, we denote for each $T\in \binom{[n]}{d}$, its corresponding monomial as $g_T(x) = \prod_{i\in T}x_i$. We compute the contribution from each monomial. Note that if $T\ne \{i_1,\dots, i_d\}$, then there exists $r\in [d]$ such that $i_r\not\in T$. Therefore summing over $u\in \F_2^d$, each value appears twice and thus contribute $0$ to the sum over $\F_2$. In the case where $T = \{i_1,\dots, i_d\}$, since $a_{i_1} = \dots = a_{i_d} = 0$, we have $g_T(a+u_1e_{i_1}+\dots + u_de_{i_d}) = u_1\dots u_d$ which takes value $1$ only when $u_1 = \dots = u_d = 1$. This proves the claim and thus each $d$-subcube contains an odd number of points from $A$.
\end{remark}

Finally, we present the proof of \cref{thm:one-point}.

\begin{proof}[Proof of \cref{thm:one-point}]
    Fix $0\le s\le 2^d$ and let $A\subseteq \F_2^n$ be a subset that achieves $\lambda^*(n,d,s,A) = \lambda^*(n,d,s)$. 

    For $2\le s\le 2^d$, consider a subset $A'\subseteq A$ obtained from randomly deleting each point in $A$ with probability $1/s$ independently. Then for a uniformly random $d$-flat $Q$, we have 
    \[\P[|A'\cap Q| = s-1 \mid |A\cap Q| = s] = s\cdot \frac{1}{s}\cdot \lpr{1-\frac{1}{s}}^{s-1}=  \lpr{1-\frac{1}{s}}^{s-1} =  \lpr{\frac{s-1}{s}}^{s-1}.\]
    Thus we have
    \[\P[|A'\cap Q| = s-1] \ge \lpr{\frac{s-1}{s}}^{s-1} \cdot \P[|A\cap Q| = s] = \lambda^*(n,d,s)\lpr{\frac{s-1}{s}}^{s-1}.\]
    Fixing an outcome of $A'$ that achieves the above inequality, we can conclude that 
    \[\lambda^*(n,d,s-1)\ge \lambda^*(n,d,s-1, A') \ge \lambda^*(n,d,s)\lpr{\frac{s-1}{s}}^{s-1}.\]
    Taking $n\to \infty$ implies the desired statement.

    Similarly, for $0\le s\le 2^d-2$, consider the set $A''\supseteq A$ obtained from randomly adding each point in $\F_2^n\setminus A$ to $A$ with probability $1/(2^d-s)$ independently. Then for a uniformly random $d$-flat $Q$, we have
    \[\P[|A''\cap Q| = s+1 \mid |A\cap Q| = s] =  \lpr{1-\frac{1}{2^d-s}}^{2^d-s-1} = \lpr{\frac{2^d-s-1}{2^d-s}}^{2^d-s-1}.\]
    Thus we have
    \[\P[|A''\cap Q| = s+1] \ge  \lpr{\frac{2^d-s-1}{2^d-s}}^{2^d-s-1} \cdot \P[|A\cap Q| = s] = \lambda^*(n,d,s)\lpr{\frac{2^d-s-1}{2^d-s}}^{2^d-s-1}.\]
    Fixing an outcome of $A''$ that achieves the above inequality, we can conclude that 
    \[\lambda^*(n,d,s+1)\ge \lambda^*(n,d,s+1, A'') \ge \lambda^*(n,d,s)\lpr{\frac{2^d-s-1}{2^d-s}}^{2^d-s-1}.\]
    Taking $n\to \infty$ implies the desired statement.
\end{proof}

\section{Concluding remarks}\label{sec:concluding-remarks}

In this paper, we study the intersection statistics of affine subspaces over $\F_2^n$ as a natural coordinate-invariant analogue of the axis-aligned cube statistics introduced by Alon, Axenovich, and Goldwasser~\cite{alon2024hypercubestatistics}. By replacing the axis-aligned $d$-subcubes with all affine $d$-flats, the problem acquires full affine symmetry. This change turns out to have substantial consequences for the behavior of intersection statistics. In this section, we present some natural follow-up questions as well as discussions for the flat statistics versus the cube statistics.

\subsection{Determine exact values for $\lambda^*(d,s)$}

Currently the only cases for $s$ where $\lambda^*(d,s)$ is completely determined are when $s\in \{0,2^{d-1},2^d\}$, so determining the exact values for $\lambda^*(d,s)$ for most values of $s$ remains an interesting open problem. In \cref{thm:s-even}, we showed that for $s = j\cdot 2^k$ with $j$ odd and $k\ge 1$, we have
\[\lambda^*(d,s) = 1- \Theta(2^{-k}),\]
which determines $1-\lambda^*(d,s)$ for all even $s$ up to an absolute constant. In particular, the lower bound is roughly $1-2^{-k}$ and the upper bound is roughly $1-\frac{2}{3}\cdot 2^{-k}$ as $d\to \infty$, so the constants in front of the $2^{-k}$ term differ by a multiplicative factor of $2/3$. It remains an interesting question to determine the exact leading constant. We conjecture that the lower bound is the correct answer.

\begin{conj}\label{conj:flat-s-even}
    For $d\ge 1$ and $s = j\cdot 2^k$ with $j$ odd and $1\le k\le d$, we have
    \[\lambda^*(d,s) = (1+o_d(1))c(d,k).\]
\end{conj}
We consider the case where $k$ is fixed and $d\to \infty$ in the above conjecture. We suspect that in order to prove \cref{conj:flat-s-even}, one would need to prove a strong structural statement that the extremizer of $\lambda^*(d,s)$ must be close to a union of $j$ flats of codimension $d-k$. Our current proof does not give any structural information on the set $A\subseteq\F_2^n$ achieving $\lambda^*(n,d,s)$, so new ideas are required.

In the case where $s$ is odd, we currently have the general bound 
\[0.2887\le \lambda^*(d,s) \le 0.5.\]
In the case where $s = 2^k\pm 1$ for $k$ large, \cref{cor:one-point-lb} implies that $\lambda^*(d,s)\ge c(d,k)/e\ge (1-2^{-k})/e$ which approaches $1/e$ as $k\to \infty$. However, there remains a big gap in the current lower bound and upper bound for all values of odd $s$. Similar to the axis-aligned setting conjectured in \cite{alon2024hypercubestatistics}, we conjecture that for $s = 1$, the optimal construction should be Poisson, which is also consistent with \cref{thm:one-point}.

\begin{conj}\label{conj:flat-s-1}
    For $d\ge 1$, we have
    \[\lambda^*(d,1) = (1+o_d(1))\cdot \frac{1}{e}.\]
\end{conj}

\subsection{Affine flat statistics versus cube statistics}

Our study also provides insights into the hypercube statistics problem. Although we showed that the 2-adic valuation of $s$ controls $\lambda^*(d,s)$, several observations imply that the same cannot be true in the axis-aligned cube setting. First, even for parity, analogous behavior fails in the axis-aligned cube setting as illustrated in \cref{sec:s-odd}. We also point out two constructions observed in \cite{alon2024hypercubestatistics} that give evidence that one should exploit the fixed coordinates in the cube statistics problem.

In \cite{alon2024hypercubestatistics}, it was observed that if one takes $A\subseteq Q_n$ consisting of every third layer of $Q_n$, then in the case where $s\in \{\ceil{2^d/3}, \floor{2^d/3}\}$, we have $\lambda(d,s)\ge 2/3-o_d(1)$. This is a significantly better lower bound than $0.2887$ in the case when $s$ is odd. Moreover, this surpasses the $1/2$ upper bound for $\lambda^*(d,s)$ for these values of $d$ and $s$, suggesting very different behaviors for $\lambda^*(d,s)$ and $\lambda(d,s)$.

Another example observed in \cite{alon2024hypercubestatistics} is the following. When $s = 2^{d-2}$, consider $A = B^{-1}(x)$ where $B: \F_2^n\to \F_2^2$ and $x\in \F_2^2$. If we take $B$ to be described by a random $2\times n$ matrix with independent nonzero column vectors in $\F_2^2$, then notice that $\lambda(n,d,s,A)$ is at least $1-3^{1-d}$ and thus we get that $\lambda(d,2^{d-2})\ge 1-3^{1-d}$. This is an asymptotically better lower bound than $\lambda^*(d,2^{d-2}) = c(d,d-2)$ which is on the order of $1 - 2^{-(d-2)}$. These examples suggest that the cube statistics is heavily dependent on the fixed coordinates and exhibits genuinely different behavior than affine flat statistics.

\printbibliography

\end{document}